\documentclass{amsart}
\usepackage{amssymb}
\usepackage{amsfonts}
\usepackage{hyperref}
\usepackage[numbers,sort&compress]{natbib}

\newtheorem{theorem}{Theorem}
\theoremstyle{plain}

\newtheorem{conjecture}{Conjecture}
\newtheorem{corollary}{Corollary}

\newtheorem{lemma}{Lemma}

\newtheorem{problem}{Problem}

\newtheorem{remark}{Remark}

\numberwithin{equation}{section}

\input{tcilatex}

\begin{document}
\title[Convexity and monotonicity for the elliptic integrals]{Convexity and
monotonicity for the elliptic integrals of the first kind and applications}
\author{Zhen-Hang Yang and Jingfeng Tian*}
\address{Zhen-Hang Yang, College of Science and Technology\\
North China Electric Power University, Baoding, Hebei Province, 071051, P.
R. China and Department of Science and Technology, State Grid Zhejiang
Electric Power Company Research Institute, Hangzhou, Zhejiang, 310014, China}
\email{yzhkm\symbol{64}163.com}
\address{Jingfeng Tian\\
College of Science and Technology\\
North China Electric Power University\\
Baoding, Hebei Province, 071051, P. R. China}
\email{tianjf\symbol{64}ncepu.edu.cn}
\thanks{*Corresponding author: Jingfeng Tian, e-mail: tianjf\symbol{64}%
ncepu.edu.cn}
\thanks{This work was supported by the Fundamental Research Funds for the
Central Universities (No. 2015ZD29) and the Higher School Science Research
Funds of Hebei Province of China (No. Z2015137).}

\begin{abstract}
The elliptic integral and its various generalizations are playing very
important and basic role in different branches of modern mathematics. It is
well known that they cannot be represented by the elementary transcendental
functions. Therefore, there is a need for sharp computable bounds for the
family of integrals. In this paper, by virtue of two new tools, we study
monotonicity and convexity of certain combinations of the complete elliptic
integrals of the first kind, and obtain new sharp bounds and inequalities
for them. In particular, we prove that the function $\mathcal{K}\left( \sqrt{%
x}\right) /\ln \left( c/\sqrt{1-x}\right) $ is concave on $\left( 0,1\right) 
$ if and only if $c=e^{4/3}$, where $\mathcal{K}$ denotes the complete
elliptic integrals of the first kind.
\end{abstract}

\maketitle


\section{Introduction}

For $r\in (0,1)$, Legendre's complete elliptic integrals $\mathcal{K}(r)$
and $\mathcal{E}\left( r\right) $ \cite{Byrd-HEIES-1971} of the first kind
and second kind are defined by%
\begin{eqnarray*}
\mathcal{K}(r) &=&\int_{0}^{\pi /2}\frac{dt}{\sqrt{1-r^{2}\sin ^{2}(t)}}, \\
\mathcal{E}(r) &=&\int_{0}^{\pi /2}\sqrt{1-r^{2}\sin ^{2}(t)}dt,
\end{eqnarray*}%
respectively. The Legendre's complete elliptic integrals play a very
important and basic role in different branches of modern mathematics such as
classical real and complex analysis, number theory, geometric function
theory, quasiconformal mappings and analysis. Motivated by the importance of
elliptic integrals, many noteworthy monotonicity and convexity properties of 
$\mathcal{K}(r)$ and $\mathcal{E}(r)$ have been obtained, for example, in 
\cite{Hancock-EI-1958,Carlson-SFAM-1977,Borwein-PA-1987,Almkvist-AMM-95-1988}%
.

As is well known, the Legendre's complete elliptic integrals cannot be
represented by the elementary transcendental functions. Therefore, there is
a need for sharp computable bounds for the family of integrals. The goal of
this paper is to study monotonicity and convexity of certain combinations of
the complete elliptic integrals of the first kind, and establish some new
bounds and inequalities for them.

The elliptic integrals occur in the formulas for the moduli of plane ring
domains in theory of quasiconformal mappings. In particular, for $r\in (0,1)$%
, the special combinations%
\begin{equation*}
\mu \left( r\right) =\frac{\pi }{2}\frac{\mathcal{K}\left( r^{\prime
}\right) }{\mathcal{K}\left( r\right) }
\end{equation*}%
the so-called modulus of the Gr\"{o}tzsch extremal ring, is indispensable in
the study of quasiconformal distortion \cite%
{Anderson-TMJ023-1971,Lehto-GMW-126-1973}.

It is known that the complete elliptic integrals are the particular cases of
the Gaussian hypergeometric function%
\begin{equation}
F(a,b;c;x)=\sum_{n=0}^{\infty }\frac{(a)_{n}(b)_{n}}{(c)_{n}}\frac{x^{n}}{n!}%
\quad (-1<x<1),  \label{Gh}
\end{equation}%
where $(a)_{n}=\Gamma (a+n)/\Gamma (a)$ and $\Gamma $ is the classical gamma
function. Indeed, we have 
\begin{equation*}
\mathcal{K}(r)=\frac{\pi }{2}F\left( \frac{1}{2},\frac{1}{2};1;r^{2}\right) 
\text{ \ and\ \ }\mathcal{E}(r)=\frac{\pi }{2}F\left( \frac{1}{2},-\frac{1}{2%
};1;r^{2}\right) .
\end{equation*}

The hypergeometric function $F(a;b;c;x)$ has the simple differentiation
formulas%
\begin{equation}
F^{\prime }\left( a,b;c;x\right) =\frac{ab}{c}F\left( a+1,b+1;c+1;x\right) .
\label{df}
\end{equation}%
The behavior of the hypergeometric function near $x=1$ in the three cases $%
a+b<c$, $a+b=c$, and $a+b>c$, $a,b,c>0$, is given by%
\begin{equation}
\left\{ 
\begin{array}{lc}
\bigskip F\left( a,b;c;1\right) =\tfrac{\Gamma \left( c\right) \Gamma \left(
c-a-b\right) }{\Gamma \left( c-a\right) \Gamma \left( c-b\right) } & \text{%
if }c>a+b\text{,} \\ 
\bigskip F\left( a,b;c;x\right) =\tfrac{R\left( a,b\right) -\ln \left(
1-x\right) }{B\left( a,b\right) }+O\left( \left( 1-x\right) \ln \left(
1-x\right) \right) & \text{if }c=a+b\text{,} \\ 
F\left( a,b;c;x\right) =(1-x)^{c-a-b}F(c-a,c-b;c;x) & \text{if }c<a+b\text{,}%
\end{array}%
\right.  \label{F-near1}
\end{equation}%
where $B\left( a,b\right) =\Gamma \left( a\right) \Gamma \left( b\right)
/\Gamma \left( a+b\right) $ is the beta function and%
\begin{equation}
R\left( a,b\right) =-2\gamma -\psi \left( a\right) -\psi \left( b\right) 
\text{,}  \label{R(a,b)}
\end{equation}%
here $\psi \left( z\right) =\Gamma ^{\prime }\left( z\right) /\Gamma \left(
z\right) $, $\Re \left( z\right) >0$ and $\gamma $ is the Euler-Mascheroni
constant (see \cite{Abramowitz-HMFFGMT-NY-1965}). In particular, from the
second formula in (\ref{F-near1}) it is easy to obtain the following
asymptotic formula%
\begin{equation*}
\mathcal{K}(r)=\frac{\pi }{2}F\left( \frac{1}{2},\frac{1}{2};1;r^{2}\right)
\thicksim \ln \frac{4}{r^{\prime }},\text{ \ as }r\rightarrow 1^{-},
\end{equation*}%
where and in what follows $r^{\prime }=\sqrt{1-r^{2}}$.

Motivated by the asymptotic formula for $\mathcal{K}(r)$, many comparison
inequalities for $\mathcal{K}(r)$ with $\ln \left( 4/r^{\prime }\right) $
were established, for example, in \cite%
{Borwein-SIAM-R-26-1984,Carlson-SIAM-JMA-16-1985,Anderson-SIAM-JMA-21-1990,Anderson-SIAM-JMA-23-1992, Qiu-JHIEE-3-1993,Kuhnau-ZAMM-74-1994,Qiu-SIAM-JMA-27(3)-1996, Anderson-CA-14-1998,Alzer-MPCPS-124(2)-1998,Qiu-SIAM-JMA-29-1998, Alzer-JCAM-172-2004,Alzer-AM-14-2015}%
. Some of new generalizations can be found in \cite%
{Baricz-MZ-256(4)-2007,Heikkala-JMAA-338(1)-2008,Heikkala-CMFT-9(1)-2009,Zhang-PRSES-A-139(2)-2009, Baricz-PAMS-142(9)-2014,Wang-MIA-16(3)-2013,Wang-JMAA-429-2015}%
.

In particular, Anderson, Vamanamurthy and Vuorinen \cite[Theorem 2.2]%
{Anderson-SIAM-JMA-21-1990}\ showed that the ratio $\mathcal{K}(r)/\ln
(c/r^{\prime })$ is strictly decreasing if and only if $0<c\leq 4$ (strictly
speaking, $c$ should be in $\left[ 1,4\right] $) and strictly increasing if
and only if $c\geq e^{2}$. It is natural to ask that what are the conditions
such that this function is concave or convex on $\left( 0,1\right) $? The
first aim of this paper is to give an answer for concavity.

\begin{theorem}
\label{MT-1}Let $c\geq 1$. The function%
\begin{equation}
x\mapsto Q_{1}\left( x\right) =\frac{\mathcal{K}\left( \sqrt{x}\right) }{\ln
\left( c/\sqrt{1-x}\right) }  \label{Q1}
\end{equation}%
is strictly concave on $\left( 0,1\right) $ if and only if $c=e^{4/3}$.
\end{theorem}

\begin{remark}
According to \cite[Theorem 2.2]{Anderson-SIAM-JMA-21-1990}, $Q_{1}$ is
decreasing on $\left( 0,1\right) $ for $c=e^{4/3}\in (0,4]$. Also, by the
properties of positive concave functions, we easily see that $Q_{1}$ is
log-concave and $1/Q_{1}$ is convex on $\left( 0,1\right) $.
\end{remark}

In 1992, Anderson, Vamanamurthy and Vuorinen \cite{Anderson-SIAM-JMA-23-1992}
conjectured that the inequality%
\begin{equation}
\mathcal{K}\left( r\right) <\ln \left( 1+\frac{4}{r^{\prime }}\right)
-\left( \ln 5-\frac{\pi }{2}\right) \left( 1-r\right)  \label{K<AVV}
\end{equation}%
holds for $r\in \left( 0,1\right) $, which was proved in \cite%
{Qiu-SIAM-JMA-29-1998} by Qiu, Vamanamurthy and Vuorinen. We remark that as
an approximation for $\mathcal{K}\left( r\right) $, the function $\ln \left(
1+4/r^{\prime }\right) $ is obviously better than $\ln \left( 4/r^{\prime
}\right) $. Therefore, the second aim of this paper is to further vestigate
the monotonicity of the ratio $\mathcal{K}\left( r\right) /\ln \left(
1+4/r^{\prime }\right) $ and the convexity of the difference $\mathcal{K}%
\left( r\right) -\ln \left( 1+4/r^{\prime }\right) $. We will prove the
following two theorems.

\begin{theorem}
\label{MT-2}The function%
\begin{equation}
r\mapsto Q_{2}\left( r\right) =\frac{\mathcal{K}\left( r\right) }{\ln \left(
1+4/r^{\prime }\right) }  \label{Q2}
\end{equation}%
is strictly increasing from $\left( 0,1\right) $ onto $\left( \pi /\ln
25,1\right) $. Consequently, the double inequality%
\begin{equation*}
\frac{\pi }{2\ln 5}\ln \left( 1+\frac{4}{r^{\prime }}\right) <\mathcal{K}%
\left( r\right) <\ln \left( 1+\frac{4}{r^{\prime }}\right)
\end{equation*}%
holds with the best constants $\pi /\ln 25\approx 0.98$ and $1$.
\end{theorem}

\begin{theorem}
\label{MT-3}The function%
\begin{equation}
D\left( x\right) =\mathcal{K}\left( \sqrt{x}\right) -\ln \left( 1+\frac{4}{%
\sqrt{1-x}}\right)  \label{D}
\end{equation}%
is strictly convex on $\left( 0,1\right) $.
\end{theorem}

\section{Lemmas}

To prove our main results, we need the following lemmas. In order to state
the first lemma, we need to introduce a useful auxiliary function $H_{f,g}$.
For $-\infty \leq a<b\leq \infty $, let $f$ and $g$ be differentiable on $%
(a,b)$ and $g^{\prime }\neq 0$ on $(a,b)$. Then the function $H_{f,g}$ is
defined by 
\begin{equation}
H_{f,g}:=\frac{f^{\prime }}{g^{\prime }}g-f.  \label{H_f,g}
\end{equation}%
The function $H_{f,g}$ has some well properties \cite[Properties 1, 2]%
{Yang-arxiv-1409.6408}. In particular, we have%
\begin{eqnarray}
\left( \frac{f}{g}\right) ^{\prime } &=&\frac{g^{\prime }}{g^{2}}\left( 
\frac{f^{\prime }}{g^{\prime }}g-f\right) =\frac{g^{\prime }}{g^{2}}H_{f,g},
\label{d-f/g} \\
H_{f,g}^{\prime } &=&\left( \frac{f^{\prime }}{g^{\prime }}\right) ^{\prime
}g.  \label{dHf,g}
\end{eqnarray}

\begin{lemma}[{\protect\cite[Theorem 1]{Yang-JMAA-428-2015}}]
\label{L-A/B-pm}Let $A\left( t\right) =\sum_{k=0}^{\infty }a_{k}t^{k}$ and $%
B\left( t\right) =\sum_{k=0}^{\infty }b_{k}t^{k}$ be two real power series
converging on $\left( -r,r\right) $ and $b_{k}>0$ for all $k$. Suppose that
for certain $m\in 
\mathbb{N}
$, the non-constant sequence $\left\{ a_{k}/b_{k}\right\} $ is increasing
(resp. decreasing) for $0\leq k\leq m$ and decreasing (resp. increasing) for 
$k\geq m$. Then the function $A/B$ is strictly increasing (resp. decreasing)
on $\left( 0,r\right) $ if and only if $H_{A,B}\left( r^{-}\right) \geq $
(resp. $\leq $)$0$. Moreover, if $H_{A,B}\left( r^{-}\right) <$ (resp. $>$)$%
0 $, then there exists $t_{0}\in \left( 0,r\right) $ such that the function $%
A/B$ is strictly increasing (resp. decreasing) on $\left( 0,t_{0}\right) $
and strictly decreasing (resp. increasing) on $\left( t_{0},r\right) $.
\end{lemma}

\begin{remark}
Lemma \ref{L-A/B-pm} is a powerful tool to deal with the monotonicity of the
ratio of power series in the case when the sequence $\{a_{k}/b_{k}\}_{k\geq
0}$ is piecewise monotonic, and is now applied preliminarily, see \cite%
{Yang-JIA-40-2016}, \cite{Wang-AMC-276-2016}, \cite{Yang-MIA-accepted}.
\end{remark}

The following lemma offers a simple criterion to determine the sign of a
class of special series. A polynomial version appeared in \cite%
{Yang-AAA-2014-702718}, and another series version converging on $\left(
0,\infty \right) $ can see \cite{Yang-JIA-299-2015}.

\begin{lemma}
\label{L-sgnS}Let $\{a_{k}\}_{k=0}^{\infty }$ be a nonnegative real sequence
with $a_{m}>0$ and $\sum_{k=m+1}^{\infty }a_{k}>0$ and let%
\begin{equation*}
S\left( t\right) =-\sum_{k=0}^{m}a_{k}t^{k}+\sum_{k=m+1}^{\infty }a_{k}t^{k}
\end{equation*}%
be a convergent power series on the interval $\left( 0,r\right) $ ($r>0$).
(i) If $S\left( r^{-}\right) \leq 0$, then $S\left( t\right) <0$ for all $%
t\in \left( 0,r\right) $. (ii) If $S\left( r^{-}\right) >0$, then there is a
unique $t_{0}\in \left( 0,r\right) $ such that $S\left( t\right) <0$ for $%
t\in \left( 0,t_{0}\right) $ and $S\left( t\right) >0$ for $t\in \left(
t_{0},r\right) $.
\end{lemma}

\begin{proof}
We prove the desired assertions by mathematical induction for the negative
integral $m$.

For $m=0$, we have $S^{\prime }\left( t\right) =\sum_{k=m+1}^{\infty
}ka_{k}t^{k-1}>0$, which together with $S\left( 0^{+}\right) =-a_{0}<0$ and $%
S\left( r^{-}\right) \leq \left( >\right) 0$ yields the desired assertions.

Suppose that the desired assertions are true for $m=n$. We prove they are
also true for $m=n+1$ by distinguishing two cases.

Case 1: $S^{\prime }\left( r^{-}\right) \leq 0$. By inductive hypothesis we
have $S^{\prime }\left( t\right) <0$ for all $t\in \left( 0,r\right) $. Then 
$S\left( t\right) <S\left( 0^{+}\right) =-a_{0}\leq 0$ for all $t\in \left(
0,r\right) $.

Case 2: $S^{\prime }\left( r^{-}\right) >0$. By inductive hypothesis it is
deduced that there is a $t_{1}\in \left( 0,r\right) $ such that $S^{\prime
}\left( t\right) <0$ for $t\in \left( 0,t_{1}\right) $ and $S^{\prime
}\left( t\right) >0$ for $t\in \left( t_{1},r\right) $. If $S\left(
r^{-}\right) \leq 0$, then we have $S\left( t\right) <S\left( r^{-}\right)
\leq 0$ for $t\in \left( t_{1},r\right) $ and $S\left( t\right) <S\left(
0^{+}\right) =-a_{0}\leq 0$ for $t\in \left( 0,t_{1}\right) $, which implies
that $S\left( t\right) \leq 0$ for all $t\in \left( 0,r\right) $. If $%
S\left( r^{-}\right) >0$, then since $S\left( t\right) <S\left( 0^{+}\right)
=-a_{0}\leq 0$ for $t\in \left( 0,t_{1}\right) $, there is a $t_{0}\in
\left( t_{1},r\right) $ such that $S\left( t\right) <0$ for $t\in \left(
0,t_{0}\right) $ and $S\left( t\right) >0$ for $t\in \left( t_{0},r\right) $%
, which completes the proof.
\end{proof}

\begin{lemma}
\label{L-af-1,2,3}We have the following asymptotic formulas:%
\begin{eqnarray*}
F\left( \frac{1}{2},\frac{1}{2};1;x\right) &=&\frac{\ln \left( 16/t\right) }{%
\pi }+\frac{t}{4\pi }\left( \ln \left( 16/t\right) -2\right) +O\left(
t^{2}\ln t\right) , \\
F\left( \frac{1}{2},\frac{1}{2};2;x\right) &=&\frac{4}{\pi }-\frac{t}{\pi }%
\left( \ln \left( 16/t\right) -3\right) +O\left( t^{2}\ln t\right) ,
\end{eqnarray*}%
\ as $t\rightarrow 0^{+}$, where $t=1-x$.
\end{lemma}

\begin{proof}
It was listed in \cite[p. 559]{Abramowitz-HMFFGMT-NY-1965} that%
\begin{equation*}
F\left( a,b;a+b;x\right) =\frac{1}{B\left( a,b\right) }\sum_{n=0}^{\infty }%
\frac{\left( a\right) _{n}\left( b\right) _{n}}{n!^{2}}\left[ 2\psi \left(
n+1\right) -\psi \left( n+a\right) -\psi \left( n+b\right) -\ln t\right]
t^{n},
\end{equation*}%
and for $m=1,2,3,...$,%
\begin{equation*}
\begin{array}{l}
F\left( a,b;a+b+m;x\right) =\dfrac{\Gamma \left( m\right) \Gamma \left(
a+b+m\right) }{\Gamma \left( a+m\right) \Gamma \left( b+m\right) }%
\sum_{n=0}^{m-1}\dfrac{\left( a\right) _{n}\left( b\right) _{n}}{n!\left(
1-m\right) _{n}}t^{n}\bigskip \\ 
-\dfrac{\Gamma \left( a+b+m\right) }{\Gamma \left( a\right) \Gamma \left(
b\right) }\left( -t\right) ^{m}\times \sum_{n=0}^{\infty }\left[ \dfrac{%
\left( a+m\right) _{n}\left( b+m\right) _{n}}{n!\left( n+m\right) !}%
t^{n}\times \right. \bigskip \\ 
\left. \left( \ln t-\psi \left( n+1\right) -\psi \left( n+m+1\right) +\psi
\left( a+n+m\right) +\psi \left( b+n+m\right) \right) \right] ,%
\end{array}%
\end{equation*}%
where $t=1-x\in \left( 0,1\right) $. Letting $a=b=1/2$ and taking $m=1,2$
yield the desired asymptotic formulas.
\end{proof}

\section{Proof of Theorem \protect\ref{MT-1}}

For convenience, we use $W_{n}$ to denote the Wallis ratio, that is,%
\begin{equation}
W_{n}=\frac{\Gamma \left( n+1/2\right) }{\Gamma \left( 1/2\right) \Gamma
\left( n+1/2\right) }.  \label{Wn}
\end{equation}%
It is easy to see that $W_{n}$ satisfies the recurrence relation%
\begin{equation}
W_{n+1}=\frac{n+1/2}{n+1}W_{n}.  \label{Wn-rr}
\end{equation}%
To prove Theorem \ref{MT-1}, we first give the following lemma.

\begin{lemma}
\label{L-b.n}Let the sequence $\{\beta _{n}\}$ be defined by%
\begin{equation}
\beta _{n}=\sum_{k=0}^{n-1}\left( \frac{\left( 11k-17\right) W_{k}^{2}}{%
\left( k+1\right) \left( k+2\right) \left( k+3\right) }\frac{1}{n-k}\right) ,
\label{b.n}
\end{equation}%
Then $\beta _{n}$ satisfies the recurrence formula%
\begin{equation}
\beta _{n+1}-\lambda _{n}\beta _{n}=-\frac{1}{9}\frac{\left( 2n+1\right)
\left( 880n^{4}+2404n^{3}-7319n^{2}-20\,301n-10\,404\right) }{\left(
11n-17\right) \left( n+1\right) ^{2}\left( n+2\right) \left( n+3\right)
\left( n+4\right) }W_{n}^{2},  \label{b.n+1-b.n}
\end{equation}%
where%
\begin{equation}
\lambda _{n}=\frac{1}{4}\dfrac{\left( 11n-6\right) \left( 2n+1\right) ^{2}}{%
\left( n+4\right) \left( 11n-17\right) \left( n+1\right) }.  \label{l.n}
\end{equation}
\end{lemma}

\begin{proof}
Denote by%
\begin{equation*}
\beta _{k,n-k}=\frac{\left( 11k-17\right) W_{k}^{2}}{\left( k+1\right)
\left( k+2\right) \left( k+3\right) }\frac{1}{n-k}.
\end{equation*}%
Then we have%
\begin{eqnarray*}
\beta _{n+1}-\lambda _{n}\beta _{n} &=&\sum_{k=0}^{n}\beta
_{k,n+1-k}-\lambda _{n}\sum_{k=0}^{n-1}\beta _{k,n-k} \\
&=&\beta _{0,n+1}+\sum_{k=0}^{n-1}\left[ \beta _{k+1,n-k}-\lambda _{n}\beta
_{k,n-k}\right]
\end{eqnarray*}%
\begin{eqnarray*}
&=&-\tfrac{17}{6\left( n+1\right) }+\sum_{k=0}^{n-1}\left( \tfrac{11k-6}{%
\left( k+4\right) \left( k+2\right) \left( k+3\right) }\left( \tfrac{k+1/2}{%
k+1}\right) ^{2}-\tfrac{\left( 11n-6\right) \left( 2n+1\right) ^{2}}{4\left(
n+4\right) \left( 11n-17\right) \left( n+1\right) }\tfrac{11k-17}{\left(
k+1\right) \left( k+2\right) \left( k+3\right) }\right) \tfrac{W_{k}^{2}}{n-k%
} \\
&=&-\tfrac{17}{6\left( n+1\right) }-\tfrac{1}{4\left( n+1\right) \left(
n+4\right) \left( 11n-17\right) }\sum_{k=0}^{n-1}\tfrac{\phi \left(
k,n\right) }{\left( k+1\right) ^{2}\left( k+2\right) \left( k+3\right)
\left( k+4\right) }W_{k}^{2},
\end{eqnarray*}%
where%
\begin{equation*}
\begin{array}{l}
\phi \left( k,n\right) =11\left( 132n^{2}-151n-266\right) k^{2}-\left(
1661n^{2}+3252n+1132\right) k-2\left( 1463n^{2}+566n-319\right) \\ 
=11\left( 132n^{2}-151n-266\right) \left( k+\frac{1}{2}\right) ^{2}-\left(
283n+299\right) \left( 11n-6\right) \left( k+\frac{1}{2}\right) -\frac{315}{4%
}\left( 2n+1\right) \left( 11n-6\right) .%
\end{array}%
\end{equation*}%
Then%
\begin{equation}
\begin{array}{l}
\beta _{n+1}-\lambda _{n}\beta _{n}=-\dfrac{17}{6\left( n+1\right) }-\dfrac{%
11}{4}\dfrac{\left( 132n^{2}-151n-266\right) }{\left( n+1\right) \left(
n+4\right) \left( 11n-17\right) }\phi _{2}\left( n\right) \bigskip \\ 
+\dfrac{1}{4}\dfrac{\left( 283n+299\right) \left( 11n-6\right) }{\left(
n+1\right) \left( n+4\right) \left( 11n-17\right) }\phi _{1}\left( n\right) +%
\dfrac{315}{16}\dfrac{\left( 2n+1\right) \left( 11n-6\right) }{\left(
n+1\right) \left( n+4\right) \left( 11n-17\right) }\phi _{0}\left( n\right) ,%
\end{array}
\label{bn-rf.}
\end{equation}%
where%
\begin{eqnarray*}
\phi _{2}\left( n\right) &=&\sum_{k=0}^{n-1}\dfrac{\left( k+1/2\right)
^{2}W_{k}^{2}}{\left( k+1\right) ^{2}\left( k+2\right) \left( k+3\right)
\left( k+4\right) }\text{, } \\
\phi _{1}\left( n\right) &=&\sum_{k=0}^{n-1}\dfrac{\left( k+1/2\right)
W_{k}^{2}}{\left( k+1\right) ^{2}\left( k+2\right) \left( k+3\right) \left(
k+4\right) }\text{, } \\
\phi _{0}\left( n\right) &=&\sum_{k=0}^{n-1}\dfrac{W_{k}^{2}}{\left(
k+1\right) ^{2}\left( k+2\right) \left( k+3\right) \left( k+4\right) }.
\end{eqnarray*}%
Note that $\phi _{i}\left( n\right) $ ($i=1,2,3$) can be written as%
\begin{eqnarray*}
\phi _{2}\left( n\right) &=&\tfrac{1}{3!}\left( \sum_{k=1}^{n}\tfrac{\left( 
\frac{1}{2}\right) _{k}^{2}}{\left( 4\right) _{k}k!}-1\right) ,\phi
_{1}\left( n\right) =-\tfrac{2}{3!}\left( \sum_{k=1}^{n}\tfrac{\left( -\frac{%
1}{2}\right) _{k}\left( \frac{1}{2}\right) _{k}}{\left( 4\right) _{k}k!}%
-1\right) , \\
\phi _{0}\left( n\right) &=&\tfrac{4}{3!}\left( \sum_{k=1}^{n}\tfrac{\left( -%
\frac{1}{2}\right) _{k}^{2}}{\left( 4\right) _{k}k!}-1\right) ,
\end{eqnarray*}%
by comparing the coefficients of power series of the following formulas%
\begin{eqnarray*}
\frac{1}{1-x}F\left( \frac{1}{2},\frac{1}{2};4;x\right) &=&\left( 1-x\right)
^{2}F\left( \frac{7}{2},\frac{7}{2};4;x\right) , \\
\frac{1}{1-x}F\left( -\frac{1}{2},\frac{1}{2};4;x\right) &=&\left(
1-x\right) ^{3}F\left( \frac{9}{2},\frac{7}{2};4;x\right) , \\
\frac{1}{1-x}F\left( -\frac{1}{2},-\frac{1}{2};4;x\right) &=&\left(
1-x\right) ^{4}F\left( \frac{9}{2},\frac{9}{2};4;x\right) ,
\end{eqnarray*}%
which follows from the second formula of (\ref{F-near1}), we can obtain%
\begin{equation*}
\phi _{2}\left( n\right) =\frac{1}{225}\frac{\left( 2n+1\right) ^{2}\left(
32n^{2}+168n+225\right) }{\left( n+1\right) \left( n+2\right) \left(
n+3\right) }W_{n}^{2}-\frac{1}{6},
\end{equation*}%
\begin{equation*}
\phi _{1}\left( n\right) =\frac{1}{3}-\frac{2}{525}\frac{\left( 2n+1\right)
\left( 128n^{3}+736n^{2}+1236n+525\right) }{\left( n+1\right) \left(
n+2\right) \left( n+3\right) }W_{n}^{2},
\end{equation*}%
\begin{equation*}
\phi _{0}\left( n\right) =\frac{4}{3675}\frac{2048n^{4}+12\,800n^{3}+25%
\,664n^{2}+18\,288n+3675}{\left( n+1\right) \left( n+2\right) \left(
n+3\right) }W_{n}^{2}-\frac{2}{3},
\end{equation*}%
here we omit the details. Substituting $\phi _{i}\left( n\right) $ ($i=1,2,3$%
) into (\ref{bn-rf.}) and simplifying prove the recurrence relation (\ref%
{b.n+1-b.n}).

The proof of this lemma is done.
\end{proof}

Now we are in a position to prove our the first result.

\begin{proof}[Proof of Theorem \protect\ref{MT-1}]
Differentiation yields%
\begin{equation*}
\frac{2}{\pi }Q_{1}^{\prime }\left( x\right) =\frac{\frac{1}{4}F\left( \frac{%
3}{2},\frac{3}{2};2;x\right) \ln \left( c/\sqrt{1-x}\right) -F\left( \frac{1%
}{2},\frac{1}{2};1;x\right) \frac{1}{2\left( 1-x\right) }}{\left( \ln \left(
c/\sqrt{1-x}\right) \right) ^{2}}:=\frac{f_{1}\left( x\right) }{g_{1}\left(
x\right) },
\end{equation*}%
\begin{eqnarray*}
\frac{2}{\pi }Q_{1}^{\prime \prime }\left( x\right) &=&\dfrac{\frac{9}{32}%
F\left( \frac{5}{2},\frac{5}{2};3;x\right) \ln \left( c/\sqrt{1-x}\right)
-F\left( \frac{1}{2},\frac{1}{2};1;x\right) \frac{1}{2\left( 1-x\right) ^{2}}%
}{\left( \ln \left( c/\sqrt{1-x}\right) \right) ^{2}} \\
&&-\frac{\left( \frac{1}{4}F\left( \frac{3}{2},\frac{3}{2};2;x\right) \ln
\left( c/\sqrt{1-x}\right) -F\left( \frac{1}{2},\frac{1}{2};1;x\right) \frac{%
1}{2\left( 1-x\right) }\right) \frac{1}{1-x}}{\left( \ln \left( c/\sqrt{1-x}%
\right) \right) ^{3}}.
\end{eqnarray*}

The necessity easily follows from the inequality%
\begin{equation*}
\frac{2}{\pi }Q_{1}^{\prime \prime }\left( 0^{+}\right) =\dfrac{1}{32}\dfrac{%
\left( 3\ln c-4\right) ^{2}}{\ln ^{3}c}\leq 0.
\end{equation*}

To prove the sufficiency, it suffices to prove $\left( 2/\pi \right)
Q_{2}^{\prime }=f_{1}/g_{1}$ is strictly decreasing on $\left( 0,1\right) $
for $c=e^{4/3}$. Differentiations by (\ref{df}) and simplifications by the
third formula of (\ref{F-near1}) yield%
\begin{eqnarray*}
\frac{f_{1}^{\prime }\left( x\right) }{g_{1}^{\prime }\left( x\right) } &=&%
\frac{\frac{9}{32}F\left( \frac{5}{2},\frac{5}{2};3;x\right) \left( \frac{4}{%
3}-\ln \sqrt{1-x}\right) -F\left( \frac{1}{2},\frac{1}{2};1;x\right) \frac{1%
}{2\left( 1-x\right) ^{2}}}{\left( \frac{4}{3}-\ln \sqrt{1-x}\right) /\left(
1-x\right) } \\
&=&\frac{\frac{9}{16}F\left( \frac{1}{2},\frac{1}{2};3;x\right) \left( \frac{%
4}{3}-\ln \sqrt{1-x}\right) -F\left( \frac{1}{2},\frac{1}{2};1;x\right) }{%
2\left( 1-x\right) \left( \frac{4}{3}-\ln \sqrt{1-x}\right) }:=\frac{%
f_{2}\left( x\right) }{g_{2}\left( x\right) }.
\end{eqnarray*}

Since $g_{2}^{\prime }=\ln \left( 1-x\right) -5/3<0$, if we prove $%
H_{f_{2},g_{2}}=\left( f_{2}^{\prime }/g_{2}^{\prime }\right) g_{2}-f_{2}>0$
for $x\in \left( 0,1\right) $, then $\left( f_{1}^{\prime }/g_{1}^{\prime
}\right) ^{\prime }=\left( f_{2}/g_{2}\right) ^{\prime }=\left(
g_{2}^{\prime }/g_{2}^{2}\right) H_{f_{2},g_{2}}<0$. This yields $%
H_{f_{1},g_{1}}^{\prime }=\left( f_{1}^{\prime }/g_{1}^{\prime }\right)
^{\prime }g_{1}<0$, and so 
\begin{equation*}
H_{f_{1},g_{1}}\left( x\right) <H_{f_{1},g_{1}}\left( 0\right) =\frac{%
f_{1}^{\prime }\left( 0\right) }{g_{1}^{\prime }\left( 0\right) }g_{1}\left(
0\right) -f_{1}\left( 0\right) =\left( -\frac{3}{32}\right) \left( \frac{4}{3%
}\right) ^{2}-\left( -\frac{1}{6}\right) =0.
\end{equation*}%
By the relation $\left( f_{1}/g_{1}\right) ^{\prime }=\left( g_{1}^{\prime
}/g_{1}^{2}\right) H_{f_{1},g_{1}}<0$, which proves $Q_{1}^{\prime \prime
}\left( x\right) <0$.

Differentiation yields%
\begin{equation*}
f_{2}^{\prime }\left( x\right) =\frac{3}{64}F\left( \frac{3}{2},\frac{3}{2}%
;4;x\right) \left( \frac{4}{3}-\ln \sqrt{1-x}\right) +\frac{9}{32}\frac{1}{%
1-x}F\left( \frac{1}{2},\frac{1}{2};3;x\right) -\frac{1}{4}F\left( \frac{3}{2%
},\frac{3}{2};2;x\right) ,
\end{equation*}%
\begin{equation*}
\frac{g_{2}\left( x\right) }{g_{2}^{\prime }\left( x\right) }=\frac{2\left(
1-x\right) \left( \frac{4}{3}-\ln \sqrt{1-x}\right) }{\ln \left( 1-x\right) -%
\frac{5}{3}}=-\left( 1-x\right) \frac{8/3-\ln \left( 1-x\right) }{5/3-\ln
\left( 1-x\right) }.
\end{equation*}%
We first claim that $f_{2}^{\prime }\left( x\right) >0$ for $x\in \left(
0,1\right) $. Due to the first item is clearly positive, it suffices to show
that the sum of the second and third ones is also positive, which is
equivalent to%
\begin{eqnarray*}
f_{3}\left( x\right) &:&=\left( 1-x\right) \left[ \frac{9}{32}\frac{1}{1-x}%
F\left( \frac{1}{2},\frac{1}{2};3;x\right) -\frac{1}{4}F\left( \frac{3}{2},%
\frac{3}{2};2;x\right) \right] \\
&=&\frac{9}{32}F\left( \frac{1}{2},\frac{1}{2};3;x\right) -\frac{1}{4}%
F\left( \frac{1}{2},\frac{1}{2};2;x\right) >0.
\end{eqnarray*}%
Expanding in power series lead us to%
\begin{equation*}
f_{3}\left( x\right) =\frac{1}{32}-\frac{1}{16}\sum_{n=1}^{\infty }\frac{4n-1%
}{\left( n+2\right) \left( n+1\right) }W_{n}^{2}x^{n},
\end{equation*}%
where $W_{n}$ is defined by (\ref{Wn}). And, we easily get that $f_{3}\left(
1^{-}\right) =0$. It follows from Lemma \ref{L-sgnS} that $f_{3}\left(
x\right) >0$ for $x\in \left( 0,1\right) $.

Second, it is obvious that%
\begin{equation*}
\frac{g_{2}\left( x\right) }{g_{2}^{\prime }\left( x\right) }=-\left(
1-x\right) \frac{8/3-\ln \left( 1-x\right) }{5/3-\ln \left( 1-x\right) }>-%
\frac{8}{5}\left( 1-x\right) .
\end{equation*}%
It is then obtained that%
\begin{equation}
H_{f_{2},g_{2}}\left( x\right) =f_{2}^{\prime }\left( x\right) \frac{%
g_{2}\left( x\right) }{g_{2}^{\prime }\left( x\right) }-f_{2}\left( x\right)
>-\frac{8}{5}\left( 1-x\right) f_{2}^{\prime }\left( x\right) -f_{2}\left(
x\right) :=f_{4}\left( x\right) ,  \label{H>f4}
\end{equation}%
and it is enough to prove $f_{4}\left( x\right) >0$ for $x\in \left(
0,1\right) $. To do this, we use series expansion to get%
\begin{eqnarray*}
f_{4}\left( x\right) &=&\left[ -\frac{3}{40}\left( 1-x\right) F\left( \frac{3%
}{2},\frac{3}{2};4;x\right) -\frac{9}{16}F\left( \frac{1}{2},\frac{1}{2}%
;3;x\right) \right] \left( \frac{4}{3}-\ln \sqrt{1-x}\right) \\
&&-\frac{8}{5}\left[ \frac{9}{32}F\left( \frac{1}{2},\frac{1}{2};3;x\right) -%
\frac{1}{4}F\left( \frac{1}{2},\frac{1}{2};2;x\right) \right] +F\left( \frac{%
1}{2},\frac{1}{2};1;x\right)
\end{eqnarray*}%
\begin{eqnarray*}
&=&\left( \sum_{n=0}^{\infty }\frac{9}{40}\frac{11n-17}{\left( n+3\right)
\left( n+2\right) \left( n+1\right) }W_{n}^{2}x^{n}\right) \left( \frac{4}{3}%
+\frac{1}{2}x\sum_{n=0}^{\infty }\frac{x^{n}}{n+1}\right) \\
&&+\sum_{n=0}^{\infty }\frac{1}{10}\frac{\left( 10n^{2}+34n+19\right)
W_{n}^{2}}{\left( n+1\right) \left( n+2\right) }x^{n}
\end{eqnarray*}%
\begin{eqnarray*}
&=&\sum_{n=0}^{\infty }\frac{1}{5}\frac{\left( 5n^{3}+32n^{2}+77n+3\right) }{%
\left( n+1\right) \left( n+2\right) \left( n+3\right) }W_{n}^{2}x^{n} \\
&&+x\sum_{n=0}^{\infty }\sum_{k=0}^{n}\left( \frac{1}{20}\frac{\left(
11k-17\right) W_{k}^{2}}{\left( k+3\right) \left( k+2\right) \left(
k+1\right) }\frac{1}{n-k+1}\right) x^{n} \\
&=&\frac{1}{10}+\sum_{n=1}^{\infty }\alpha _{n}x^{n},
\end{eqnarray*}%
where%
\begin{equation*}
\alpha _{n}=\frac{1}{5}\frac{5n^{3}+32n^{2}+77n+3}{\left( n+1\right) \left(
n+2\right) \left( n+3\right) }W_{n}^{2}+\frac{9}{80}\beta _{n},
\end{equation*}%
here $\beta _{n}$ is defined by (\ref{b.n}).

Straightforward computations give%
\begin{equation*}
\alpha _{1}=-\frac{3}{40}\text{, }\alpha _{2}=-\frac{9}{640}\text{, }\alpha
_{3}=-\frac{31}{20\,480}\text{, }\alpha _{4}=\frac{243}{163\,840}>0.
\end{equation*}%
On the other hand, by Lemma \ref{L-b.n} and recurrence relation (\ref{Wn-rr}%
), we have%
\begin{eqnarray*}
\alpha _{n+1}-\lambda _{n}\alpha _{n} &=&\tfrac{1}{5}\tfrac{\left(
5n^{3}+47n^{2}+156n+117\right) }{\left( n+2\right) \left( n+3\right) \left(
n+4\right) }\left( \tfrac{n+1/2}{n+1}\right) ^{2}W_{n}^{2} \\
&&-\tfrac{1}{4}\tfrac{\left( 11n-6\right) \left( 2n+1\right) ^{2}}{\left(
n+4\right) \left( 11n-17\right) \left( n+1\right) }\tfrac{1}{5}\tfrac{%
5n^{3}+32n^{2}+77n+3}{\left( n+1\right) \left( n+2\right) \left( n+3\right) }%
W_{n}^{2}+\tfrac{9}{80}\left( \beta _{n+1}-\lambda _{n}\beta _{n}\right)
\end{eqnarray*}%
\begin{eqnarray*}
&=&\tfrac{1}{20}\tfrac{\left( 2n+1\right) ^{2}\left(
110n^{3}+262n^{2}-936n-1971\right) }{\left( 11n-17\right) \left( n+1\right)
^{2}\left( n+2\right) \left( n+3\right) \left( n+4\right) }W_{n}^{2}-\tfrac{1%
}{80}\tfrac{\left( 2n+1\right) \left(
880n^{4}+2404n^{3}-7319n^{2}-20\,301n-10\,404\right) }{\left( 11n-17\right)
\left( n+1\right) ^{2}\left( n+2\right) \left( n+3\right) \left( n+4\right) }%
W_{n}^{2} \\
&=&\frac{3}{80}\frac{\left( 2n+1\right) \left(
44n^{3}+293n^{2}+263n+840\right) }{\left( 11n-17\right) \left( n+1\right)
^{2}\left( n+2\right) \left( n+3\right) \left( n+4\right) }W_{n}^{2}>0,
\end{eqnarray*}%
for $n\geq 2$. This in combination with $\alpha _{4}>0$ reveals that $\alpha
_{n}>0$ for $n\geq 4$. Hence, we derive that%
\begin{eqnarray}
f_{4}\left( x\right) &=&\frac{1}{10}+\sum_{n=1}^{\infty }\alpha _{n}x^{n}>%
\frac{1}{10}+\sum_{n=1}^{3}\alpha _{n}x^{n}  \label{f4>f5} \\
&=&\frac{1}{10}-\frac{3}{40}x-\frac{9}{640}x^{2}-\frac{31}{20\,480}%
x^{3}:=f_{5}\left( x\right) .  \notag
\end{eqnarray}%
Since $f_{5}\left( 1\right) =193/20\,480>0$, by Lemma \ref{L-sgnS} we get
that $f_{5}\left( x\right) >0$. Taking into account (\ref{H>f4}) and (\ref%
{f4>f5}) proves $H_{f_{2},g_{2}}\left( x\right) >0$ for $x\in \left(
0,1\right) $, which completes the proof.
\end{proof}

By the properties of the concave functions we find that%
\begin{equation*}
\sqrt{Q_{1}\left( t\right) Q_{1}\left( 1-t\right) }\leq \frac{Q_{1}\left(
t\right) +Q_{1}\left( 1-t\right) }{2}\leq Q_{1}\left( \frac{1}{2}\right) ,
\end{equation*}%
which by putting $t=r^{2}\in \left( 0,1\right) $ gives the following
assertions.

\begin{corollary}
The inequality%
\begin{equation*}
\frac{\mathcal{K}\left( r\right) }{\ln \left( e^{4/3}/r^{\prime }\right) }+%
\frac{\mathcal{K}\left( r^{\prime }\right) }{\ln \left( e^{4/3}/r\right) }%
\leq 2c_{0},
\end{equation*}%
or equivalently,%
\begin{equation}
\mu \left( r\right) \leq \pi c_{0}\frac{\ln \left( e^{4/3}/r\right) }{%
\mathcal{K}\left( r\right) }-\frac{\pi }{2}\frac{\ln \left( e^{4/3}/r\right) 
}{\ln \left( e^{4/3}/r^{\prime }\right) }  \label{Mi-1}
\end{equation}%
holds for $r\in \left( 0,1\right) $, where $c_{0}=\frac{3\Gamma \left(
1/4\right) ^{2}}{2\left( 3\ln 2+8\right) \sqrt{\pi }}\approx 1.103\,7$ is
the best constant.
\end{corollary}

\begin{corollary}
The inequality 
\begin{equation}
\mathcal{K}\left( r\right) \mathcal{K}\left( r^{\prime }\right) \leq
c_{0}^{2}\ln \left( \frac{e^{4/3}}{r}\right) \ln \left( \frac{e^{4/3}}{%
r^{\prime }}\right)  \label{Mi-2}
\end{equation}%
holds for $r\in \left( 0,1\right) $ with the best constant $c_{0}=\frac{%
3\Gamma \left( 1/4\right) ^{2}}{2\left( 3\ln 2+8\right) \sqrt{\pi }}\approx
1.103\,7$.
\end{corollary}

\begin{corollary}
\label{C-K-lc}The function $\left( 1-x\right) ^{1/4}\mathcal{K}\left( \sqrt{x%
}\right) $ is strictly log-concave on $\left( 0,1\right) $, and therefore,
we have%
\begin{equation}
\sqrt{rr^{\prime }}\mathcal{K}\left( r\right) \mathcal{K}\left( r^{\prime
}\right) \leq \frac{1}{\sqrt{2}}\mathcal{K}\left( \frac{1}{\sqrt{2}}\right)
^{2}\approx 2.430\,7.  \label{rr'KK'<}
\end{equation}
\end{corollary}

\begin{proof}
We easily see that $Q_{1}\left( x\right) $ is also log-concave on $\left(
0,1\right) $ by Theorem \ref{MT-1}. Moreover, since%
\begin{equation*}
\frac{d^{2}}{dx^{2}}\left( \left( 1-x\right) ^{1/4}\ln \frac{e^{4/3}}{\sqrt{%
1-x}}\right) =\frac{3}{32}\frac{\ln \left( 1-x\right) }{\left( 1-x\right)
^{7/4}}<0,
\end{equation*}%
the function $\left( 1-x\right) ^{1/4}\ln \left( e^{4/3}/\sqrt{1-x}\right) $
is positive and concave on $\left( 0,1\right) $, and is also log-concave on $%
\left( 0,1\right) $. Consequently, the function%
\begin{equation*}
\left( 1-x\right) ^{1/4}\mathcal{K}\left( \sqrt{x}\right) =\left( 1-x\right)
^{1/4}\ln \left( \frac{e^{4/3}}{\sqrt{1-x}}\right) \times Q_{1}\left(
x\right)
\end{equation*}%
is log-concave on $\left( 0,1\right) $. Therefore, we obtain%
\begin{equation*}
\left( 1-x\right) ^{1/4}\mathcal{K}\left( \sqrt{x}\right) x^{1/4}\mathcal{K}%
\left( \sqrt{1-x}\right) \leq \left[ \left( \frac{1}{2}\right) ^{1/4}%
\mathcal{K}\left( \sqrt{1/2}\right) \right] ^{2},
\end{equation*}%
which proves the inequality (\ref{rr'KK'<}).
\end{proof}

\begin{remark}
It was proved in \cite[Corollary 3.13]{Anderson-SIAM-JMA-23-1992} that for $%
r\in \left( 0,1\right) $,%
\begin{equation}
\frac{2}{\pi }rr^{\prime }\mathcal{K}\left( r\right) \mathcal{K}\left(
r^{\prime }\right) <\min \left( r\ln \frac{4}{r},r^{\prime }\ln \frac{4}{%
r^{\prime }}\right) =\left\{ 
\begin{array}{cc}
r\ln \dfrac{4}{r} & \text{if }r\in \left( 0,\frac{1}{\sqrt{2}}\right) , \\ 
r^{\prime }\ln \dfrac{4}{r^{\prime }} & \text{if }r\in \left( \frac{1}{\sqrt{%
2}},1\right) .%
\end{array}%
\right.  \label{KK<A}
\end{equation}%
We write inequality (\ref{Mi-2}) as%
\begin{equation}
\frac{2}{\pi }rr^{\prime }\mathcal{K}\left( r\right) \mathcal{K}\left(
r^{\prime }\right) \leq \frac{2}{\pi }c_{0}^{2}rr^{\prime }\ln \left( \frac{%
e^{4/3}}{r}\right) \ln \left( \frac{e^{4/3}}{r^{\prime }}\right) .
\label{KK<Y}
\end{equation}%
Numeric computation shows that the upper bounds in (\ref{KK<A}) and in (\ref%
{KK<Y}) are not comparable. But the minimum upper bound given in (\ref{KK<A}%
) equals $\left( 5\sqrt{2}\ln 2\right) /4\approx 1.\,2253$, which is greater
than one given in (\ref{KK<Y}), because%
\begin{equation*}
\frac{d^{2}}{dx^{2}}\sqrt{x}\ln \left( \frac{e^{4/3}}{\sqrt{x}}\right) =-%
\frac{8+3\ln \left( 1/x\right) }{24x^{3/2}}<0,
\end{equation*}%
$x\mapsto \sqrt{x}\ln \left( e^{4/3}/\sqrt{x}\right) $ is concave, which
also implies that it is log-concave on $\left( 0,1\right) $, and so%
\begin{equation*}
\frac{2}{\pi }c_{0}^{2}rr^{\prime }\ln \left( \frac{e^{4/3}}{r}\right) \ln
\left( \frac{e^{4/3}}{r^{\prime }}\right) \leq \frac{2}{\pi }c_{0}^{2}\left[ 
\sqrt{1/2}\ln \left( \frac{e^{4/3}}{\sqrt{1/2}}\right) \right] ^{2}=\frac{%
\Gamma \left( 1/4\right) ^{4}}{16\pi ^{2}}\approx 1.\,0942.
\end{equation*}
\end{remark}

\begin{remark}
It was proved \cite[Theorem 2.2 (3)]{Anderson-SIAM-JMA-21-1990} $r^{\prime }%
\mathcal{K}\left( r\right) ^{2}$ is strictly decreasing from on $[0,1)$.
This in combination with Corollary \ref{C-K-lc} shows that $\left(
1-x\right) ^{1/4}\mathcal{K}\left( \sqrt{x}\right) $ is strictly decreasing
and log-concave on $\left( 0,1\right) $.
\end{remark}

Moreover, it is clear that the functions $\left[ Q_{1}\left( t\right)
-Q_{1}\left( 0\right) \right] /t$ and $\left[ Q_{1}\left( 1\right)
-Q_{1}\left( t\right) \right] /\left( 1-t\right) $ are decreasing on $\left(
0,1\right) $, which yield the following corollary.

\begin{corollary}
Both the functions%
\begin{eqnarray*}
r &\mapsto &\frac{1}{r^{2}}\left( \frac{\mathcal{K}\left( r\right) }{\ln
\left( e^{4/3}/r^{\prime }\right) }-\frac{3\pi }{8}\right) , \\
r &\mapsto &\frac{1}{\left( r^{\prime }\right) ^{2}}\left( 1-\frac{\mathcal{K%
}\left( r\right) }{\ln \left( e^{4/3}/r^{\prime }\right) }\right)
\end{eqnarray*}%
are strictly decreasing on $\left( 0,1\right) $. Consequently, the double
inequality%
\begin{equation}
1+\left( \frac{3\pi }{8}-1\right) \left( r^{\prime }\right) ^{2}<\frac{%
\mathcal{K}\left( r\right) }{\ln \left( e^{4/3}/r^{\prime }\right) }<\frac{%
21\pi }{64}+\frac{3\pi }{64}\left( r^{\prime }\right) ^{2}  \label{Mi-3}
\end{equation}%
holds for $r\in \left( 0,1\right) $. The lower and upper bounds are sharp.
\end{corollary}

\begin{remark}
Inequalities (\ref{Mi-3}) offer a new type of sharp lower and upper bounds.
Similar inequalities can be found in \cite{Anderson-SIAM-JMA-21-1990}, \cite%
{Qiu-SIAM-JMA-27(3)-1996}. \cite{Alzer-MPCPS-124(2)-1998}.
\end{remark}

\section{Proofs of Theorems 2 and 3}

\begin{proof}[Proof of Theorem \protect\ref{MT-2}]
Let $f\left( x\right) =F\left( \frac{1}{2},\frac{1}{2};1;x\right) $ and $%
g\left( x\right) =\ln \left( 1+4/\sqrt{1-x}\right) $, where $x=r^{2}\in
\left( 0,1\right) $. Differentiations yield%
\begin{eqnarray*}
f^{\prime }\left( x\right) &=&\frac{1}{4}F\left( \frac{3}{2},\frac{3}{2}%
;2;x\right) =\frac{1}{4}\frac{1}{1-x}F\left( \frac{1}{2},\frac{1}{2}%
;2;x\right) , \\
g^{\prime }\left( x\right) &=&\frac{2}{\left( 4+\sqrt{1-x}\right) \left(
1-x\right) }=\frac{2\left( 4-\sqrt{1-x}\right) }{\left( 15+x\right) \left(
1-x\right) },
\end{eqnarray*}%
and then, we have%
\begin{equation*}
\frac{f^{\prime }\left( x\right) }{g^{\prime }\left( x\right) }=\frac{\frac{1%
}{4}\frac{1}{1-x}F\left( \frac{1}{2},\frac{1}{2};2;x\right) }{\frac{2\left(
4-\sqrt{1-x}\right) }{\left( 15+x\right) \left( 1-x\right) }}=\frac{\left(
15+x\right) F\left( \frac{1}{2},\frac{1}{2};2;x\right) }{8\left( 4-\sqrt{1-x}%
\right) }:=\frac{f_{1}\left( x\right) }{g_{1}\left( x\right) }.
\end{equation*}%
Expanding in power series leads to%
\begin{equation*}
\frac{f_{1}\left( x\right) }{g_{1}\left( x\right) }=\frac{%
15+\sum_{n=1}^{\infty }\frac{64n^{2}-56n+15}{\left( 2n-1\right) ^{2}\left(
n+1\right) }W_{n}^{2}x^{n}}{24+\sum_{n=1}^{\infty }\frac{4\Gamma \left(
n-1/2\right) }{\Gamma \left( 1/2\right) \Gamma \left( n+1\right) }x^{n}}:=%
\frac{\sum_{n=0}^{\infty }a_{n}x^{n}}{\sum_{n=0}^{\infty }b_{n}x^{n}},
\end{equation*}%
where%
\begin{eqnarray*}
a_{n} &=&\left\{ 
\begin{array}{ll}
15 & \text{for }n=0, \\ 
\dfrac{64n^{2}-56n+15}{\left( 2n-1\right) ^{2}\left( n+1\right) }W_{n}^{2} & 
\text{for }n\geq 1,%
\end{array}%
\right. \\
b_{n} &=&\left\{ 
\begin{array}{lc}
24 & \text{for }n=0, \\ 
\dfrac{8}{2n-1}W_{n} & \text{for }n\geq 1,%
\end{array}%
\right.
\end{eqnarray*}%
here $W_{n}$ is the Wallis ratio defined by (\ref{Wn}). It is easy to check
that the sequence $\{a_{n}/b_{n}\}_{n\geq 0}$ is increasing for $n=0,1$ and
decreasing for $n\geq 1$. In fact, we have%
\begin{equation*}
\frac{a_{n}}{b_{n}}=\left\{ 
\begin{array}{ll}
\dfrac{5}{8} & \text{for }n=0, \\ 
\dfrac{1}{8}\dfrac{64n^{2}-56n+15}{\left( 2n-1\right) \left( n+1\right) }%
W_{n} & \text{for }n\geq 1.%
\end{array}%
\right.
\end{equation*}%
It then follows that%
\begin{equation*}
\frac{a_{1}}{b_{1}}-\frac{a_{0}}{b_{0}}=\frac{23}{32}-\frac{5}{8}=\frac{3}{32%
}>0,
\end{equation*}%
and for $n\geq 1$,%
\begin{equation*}
\frac{a_{n+1}}{b_{n+1}}\left/ \frac{a_{n}}{b_{n}}\right. -1=-\frac{1}{2}%
\frac{64n^{2}-168n+83}{\left( n+2\right) \left( 64n^{2}-56n+15\right) }<0,
\end{equation*}%
which proves the piecewise monotonicity of the sequence $\{a_{n}/b_{n}\}_{n%
\geq 0}$.

On the other hand, we find that%
\begin{eqnarray*}
H_{f_{1},g_{1}}\left( x\right) &=&\frac{f_{1}^{\prime }\left( x\right) }{%
g_{1}^{\prime }\left( x\right) }g_{1}\left( x\right) -f_{1}\left( x\right) \\
&=&\frac{F\left( \frac{1}{2},\frac{1}{2};2;x\right) +\left( 15+x\right) 
\frac{1}{8}F\left( \frac{3}{2},\frac{3}{2}3;x\right) }{\frac{4}{\sqrt{1-x}}}
\\
&&\times 8\left( 4-\sqrt{1-x}\right) -\left( 15+x\right) F\left( \frac{1}{2},%
\frac{1}{2};2;x\right) \\
&=&-\frac{64}{\pi }<0\text{ ad }x\rightarrow 1^{-}.
\end{eqnarray*}%
From Lemma \ref{L-A/B-pm} it is deduced that there is a $x_{1}\in \left(
0,1\right) $ such that $f_{1}/g_{1}=f^{\prime }/g^{\prime }$ is increasing
on $\left( 0,x_{1}\right) $ and decreasing on $\left( x_{1},1\right) $.
Using the formulas (\ref{d-f/g}) and (\ref{dHf,g}) together with $g>0$ we
find that $H_{f,g}^{\prime }>0$ for $x\in \left( 0,x_{1}\right) $ and $%
H_{f,g}^{\prime }<0$ for $x\in \left( x_{1},1\right) $. Now, due to $%
g^{\prime }>0$, if we show that $H_{f,g}\left( 0^{+}\right) \geq 0$ and $%
H_{f,g}\left( 1^{-}\right) \geq 0$, then we have $\left( f/g\right) ^{\prime
}=\left( g^{\prime }/g^{2}\right) H_{f,g}>0$, and the proof is done. A
simple computation yields%
\begin{eqnarray*}
H_{f,g}\left( x\right) &=&\frac{f^{\prime }\left( x\right) }{g^{\prime
}\left( x\right) }g\left( x\right) -f\left( x\right) \\
&=&\frac{\left( 15+x\right) F\left( \frac{1}{2},\frac{1}{2};2;x\right) }{%
8\left( 4-\sqrt{1-x}\right) }\ln \left( 1+\frac{4}{\sqrt{1-x}}\right)
-F\left( \frac{1}{2},\frac{1}{2};1;x\right) \\
&\rightarrow &\frac{15}{24}\ln 5-1>0\text{ as }x\rightarrow 0^{+}\text{.}
\end{eqnarray*}

Application of asymptotic formulas given in Lemma \ref{L-af-1,2,3} gives%
\begin{eqnarray*}
H_{f,g}\left( x\right) &\thicksim &\left[ \frac{\left( 16-t\right) \left( 
\frac{4}{\pi }-\frac{t}{\pi }\left( \ln \left( 16/t\right) -3\right) \right) 
}{8\left( 4-\sqrt{t}\right) }\ln \left( 1+\frac{4}{\sqrt{t}}\right) \right.
\\
&&\left. -\left( \frac{\ln \left( 16/t\right) }{\pi }+\frac{t}{4\pi }\left(
\ln \left( 16/t\right) -2\right) \right) \right] \\
&\rightarrow &0\text{ as }t=1-x\rightarrow 0^{+}\text{.}
\end{eqnarray*}

This completes the proof.
\end{proof}

Using the monotonicity of $Q_{2}$, we obtain the following Corollary.

\begin{corollary}
The inequality%
\begin{equation*}
\mu =\frac{\pi }{2}\frac{\mathcal{K}\left( r^{\prime }\right) }{\mathcal{K}%
\left( r\right) }>\frac{\pi }{2}\frac{\ln \left( 1+4/r\right) }{\ln \left(
1+4/r^{\prime }\right) }
\end{equation*}%
holds for $r\in \left( 0,1/\sqrt{2}\right) $. It is reversed for $r\in
\left( 1/\sqrt{2},1\right) $.
\end{corollary}

\begin{proof}[Proof of Theorem \protect\ref{MT-3}]
Differentiations yield%
\begin{eqnarray*}
D^{\prime }\left( x\right) &=&\frac{\pi }{8}F\left( \frac{3}{2},\frac{3}{2}%
;2;x\right) -\frac{2\left( 4-\sqrt{1-x}\right) }{\left( 15+x\right) \left(
1-x\right) }, \\
D^{\prime \prime }\left( x\right) &=&\frac{\pi }{8}\frac{9}{8}F\left( \frac{5%
}{2},\frac{5}{2};3;x\right) -\frac{16\left( x+7\right) -\left( 3x+13\right) 
\sqrt{1-x}}{\left( x+15\right) ^{2}\left( 1-x\right) ^{2}} \\
&:&=\frac{h\left( x\right) }{\left( x+15\right) ^{2}\left( 1-x\right) ^{2}},
\end{eqnarray*}%
where%
\begin{equation*}
h\left( x\right) =\frac{9\pi }{64}\left( x+15\right) ^{2}F\left( \frac{1}{2},%
\frac{1}{2};3;x\right) +\left( 3x+13\right) \sqrt{1-x}-16\left( x+7\right) .
\end{equation*}%
Expanding in power series yields%
\begin{eqnarray*}
h\left( x\right) &=&\frac{9}{32}\left( x+15\right) ^{2}\sum_{n=0}^{\infty }%
\frac{W_{n}^{2}}{\left( n+1\right) \left( n+2\right) }x^{n} \\
&&-\left( 3x+13\right) \sum_{n=0}^{\infty }\frac{W_{n}}{2n-1}x^{n}-16\left(
x+7\right)
\end{eqnarray*}%
\begin{eqnarray*}
&=&\left( \frac{2025}{64}\pi -99\right) +\left( \frac{1755}{256}\pi -\frac{39%
}{2}\right) x \\
&&+\sum_{n=2}^{\infty }\left( \dfrac{9}{32}\dfrac{4096n^{4}-14\,848n^{3}+17%
\,984n^{2}-8672n+2025}{\left( 2n-1\right) ^{2}\left( 2n-3\right) ^{2}\left(
n+1\right) \left( n+2\right) }W_{n}^{2}\right.
\end{eqnarray*}%
\begin{equation*}
\left. -\dfrac{32n-39}{\left( 2n-1\right) \left( 2n-3\right) }W_{n}\right)
x^{n}:=\sum_{n=0}^{\infty }c_{n}x^{n},
\end{equation*}%
where%
\begin{equation*}
c_{0}=\left( \frac{2025}{64}\pi -99\right) >0\text{, \ }c_{1}=\left( \frac{%
1755}{256}\pi -\frac{39}{2}\right) >0\text{,}
\end{equation*}%
and for $n\geq 2$,%
\begin{equation*}
c_{n}=\frac{\left( 32n-39\right) W_{n}}{\left( 2n-1\right) \left(
2n-3\right) }\left( d_{n}-1\right) ,
\end{equation*}%
here%
\begin{equation*}
d_{n}=\frac{9\sqrt{\pi }}{32}\frac{4096n^{4}-14\,848n^{3}+17%
\,984n^{2}-8672n+2025}{\left( 2n-1\right) \left( 2n-3\right) \left(
n+1\right) \left( n+2\right) \left( 32n-39\right) }\frac{\Gamma \left(
n+1/2\right) }{\Gamma \left( n+1\right) }.
\end{equation*}%
Clearly, in view of $\lim_{n\rightarrow \infty }\left( d_{n}-1\right) =-1$,
we see that all coefficients $c_{n}$ do not keep the same sign. However, if
we prove that there is a $n_{0}\geq 2$ such that $\left( d_{n}-1\right) \geq
0$ for $2\leq n\leq n_{0}$ and $\left( d_{n}-1\right) <0$ for $n>n_{0}$, and 
$h\left( 1^{-}\right) \geq 0$, then by Lemma \ref{L-sgnS} we obtain that $%
h\left( x\right) \geq 0$ for all $x\in \left( 0,1\right) $, and so is $%
D^{\prime \prime }\left( x\right) $. Now, since the numerator of the second
member of the expression of $d_{n}$ is positive for $n\geq 2$, since it can
be written as%
\begin{equation*}
16n^{2}\left( 16n-29\right) ^{2}+\left( 4528n-8672\right) n+2025>0,
\end{equation*}%
we easily check that%
\begin{equation*}
\frac{d_{n+1}}{d_{n}}-1=-\dfrac{196\,608\times P_{5}\left( n\right) }{\left(
32n-7\right) \left( n+3\right) \left(
4096n^{4}-14\,848n^{3}+17\,984n^{2}-8672n+2025\right) }<0,
\end{equation*}%
because%
\begin{eqnarray*}
P_{5}\left( n\right) &=&n^{5}-\frac{427}{96}n^{4}+\frac{1823}{256}n^{3}-%
\frac{33\,203}{6144}n^{2}+\frac{4831}{2048}n-\frac{51\,165}{131\,072} \\
&=&\left( n-2\right) ^{5}+\frac{533}{96}\left( n-2\right) ^{4}+\frac{8861}{%
768}\left( n-2\right) ^{3} \\
&&+\frac{64\,957}{6144}\left( n-2\right) ^{2}+\frac{23\,729}{6144}\left(
n-2\right) +\frac{67\,235}{131\,072}
\end{eqnarray*}%
is positive for $n\geq 2$.

On the other hand, it is easy to verify that%
\begin{eqnarray*}
h\left( x\right) &=&\frac{9\pi }{64}\left( x+15\right) ^{2}F\left( \frac{1}{2%
},\frac{1}{2};3;x\right) +\left( 3x+13\right) \sqrt{1-x}-16\left( x+7\right)
\\
&\rightarrow &\frac{9\pi }{64}\left( 1+15\right) ^{2}\frac{32}{9\pi }%
-16\left( 1+7\right) =0\text{ as }x\rightarrow 1^{-}.
\end{eqnarray*}%
We thus complete the proof.
\end{proof}

Applying the properties of the convex functions to $D\left( x\right) $, we
can obtain the following results.

\begin{corollary}
We have%
\begin{equation*}
\frac{\mathcal{K}\left( r\right) +\mathcal{K}\left( r^{\prime }\right) }{2}>%
\frac{\ln \left( 1+4/r^{\prime }\right) +\ln \left( 1+4/r\right) }{2}+%
\mathcal{K}\left( \frac{1}{\sqrt{2}}\right) -\ln \left( 1+4\sqrt{2}\right)
\end{equation*}%
holds for $r\in \left( 0,1\right) $.
\end{corollary}

\begin{corollary}
Both the functions%
\begin{eqnarray*}
r &\mapsto &\frac{\mathcal{K}\left( r\right) -\ln \left( 1+4/r^{\prime
}\right) -\left( \ln 5-\pi /2\right) }{r^{2}}, \\
r &\mapsto &\frac{\mathcal{K}\left( r\right) -\ln \left( 1+4/r^{\prime
}\right) }{\left( r^{\prime }\right) ^{2}}
\end{eqnarray*}%
are strictly increasing and decreasing on $\left( 0,1\right) $,
respectively. And consequently, the double inequality%
\begin{equation}
\ln \left( 1+\frac{4}{r^{\prime }}\right) -\left( \ln 5-\frac{\pi }{2}%
\right) +pr^{2}<\mathcal{K}\left( r\right) <\ln \left( 1+\frac{4}{r^{\prime }%
}\right) -\left( \ln 5-\frac{\pi }{2}\right) +qr^{2}  \label{K<>}
\end{equation}%
holds for $r\in \left( 0,1\right) $ if and only if $p\leq p_{0}=\pi /8-2/5$
and $q\geq p_{1}=\ln 5-\pi /2$.
\end{corollary}

\begin{remark}
When $q=p_{1}=\ln 5-\pi /2$, the right hand side inequality of (\ref{K<>})
can be written as%
\begin{equation*}
\mathcal{K}\left( r\right) <\ln \left( 1+\frac{4}{r^{\prime }}\right)
-\left( \ln 5-\frac{\pi }{2}\right) \left( 1-r^{2}\right) <\ln \left( 1+%
\frac{4}{r^{\prime }}\right) -\left( \ln 5-\frac{\pi }{2}\right) \left(
1-r\right) ,
\end{equation*}%
which gives another proof of the conjecture posed in \cite%
{Anderson-SIAM-JMA-23-1992}, or inequality (\ref{K<AVV}). Moreover,
inequality (\ref{K<>}) gives a companion one.
\end{remark}

\begin{remark}
Due to $D^{\prime }\left( 0\right) =\pi /8-2/5<0$, $D^{\prime }\left(
1^{-}\right) =\infty $, there is a $x_{0}\in \left( 0,1\right) $ such that $%
D^{\prime }\left( x\right) <0$ for $x\in \left( 0,x_{0}\right) $ and $%
D^{\prime }\left( x\right) >0$ for $x\in \left( x_{0},1\right) $.
\end{remark}

\section{Concluding remarks}

In this paper, we investigate the convexity of two functions $Q_{1}\left(
x\right) =\mathcal{K}\left( \sqrt{x}\right) $/$\ln \left( c/\sqrt{1-x}%
\right) $ and $D\left( x\right) =\mathcal{K}\left( \sqrt{x}\right) -\ln
\left( 1+4/\sqrt{1-x}\right) $, and monotonicity of $Q_{2}\left( r\right) =%
\mathcal{K}\left( r\right) /\ln \left( 1+4/r^{\prime }\right) $. It is known
that $Q_{1}$ is decreasing (increasing) on $\left( 0,1\right) $ if and only
if $c\in \left[ 1,4\right] $ ($[e^{2},\infty )$), but $Q_{1}$ is concave on $%
\left( 0,1\right) $ if and only if $c=e^{4/3}$, which is an interesting
result. We also show that $Q_{2}$ and $D$ are strictly increasing and convex
on $\left( 0,1\right) $, respectively. These results give some new sharp
inequalities for $\mathcal{K}\left( r\right) $, $\mathcal{K}\left( r^{\prime
}\right) $ and their combinations $\mu \left( r\right) $. Moreover, it
should be noted that these functions mentioned above are simple but it is
somewhat difficult to prove their properties, and fortunately, we have
Lemmas \ref{L-A/B-pm} and \ref{L-sgnS} as powerful tools.

Finally, we close this paper by posing the following problem and conjectures.

\begin{problem}
Determine the best parameter $c$ such that $Q_{1}\left( x\right) =\mathcal{K}%
\left( \sqrt{x}\right) /\ln \left( c/\sqrt{1-x}\right) $ is convex on $%
\left( 0,1\right) $, or $1/Q_{1}$ is concave on $\left( 0,1\right) $.
\end{problem}

\begin{conjecture}
The function $x\mapsto \left( 1-x\right) ^{p}\mathcal{K}\left( \sqrt{x}%
\right) $ is log-concave on $\left( 0,1\right) $ if and only if $p\geq 7/32$.
\end{conjecture}

\begin{conjecture}
The function $Q_{2}\left( x\right) =\mathcal{K}\left( \sqrt{x}\right) /\ln
\left( 1+4/\sqrt{1-x}\right) $ is convex on $\left( 0,1\right) $, or $%
1/Q_{2} $ is concave on $\left( 0,1\right) $.
\end{conjecture}

\begin{conjecture}
Let $D\left( x\right) =\mathcal{K}\left( \sqrt{x}\right) -\ln \left( 1+4/%
\sqrt{1-x}\right) $. Then $D^{\prime \prime }$ is absolute monotonic on $%
\left( 0,1\right) $.
\end{conjecture}

\begin{remark}
As pointed out in \cite{Anderson-CA-14-1998}, the monotonicity of a function
normally yields constant bounds for this function, while the convexity or
concavity yield better estimates. If these problem and conjectures are
proved to be true, then it will yields some new better estimates for $%
\mathcal{K}\left( r\right) $, $\mathcal{K}\left( r^{\prime }\right) $ and
their combinations. For example, if $1/Q_{2}$ is concave on $\left(
0,1\right) $, then we have%
\begin{eqnarray*}
\sqrt{\frac{\ln \left( 1+4/r^{\prime }\right) }{\mathcal{K}\left( r\right) }%
\frac{\ln \left( 1+4/r\right) }{\mathcal{K}\left( r^{\prime }\right) }}
&\leq &\frac{1}{2}\left( \frac{\ln \left( 1+4/r^{\prime }\right) }{\mathcal{K%
}\left( r\right) }+\frac{\ln \left( 1+4/r\right) }{\mathcal{K}\left(
r^{\prime }\right) }\right) \leq \frac{\ln \left( 1+4\sqrt{2}\right) }{%
\mathcal{K}\left( 1/\sqrt{2}\right) }, \\
1-\frac{2\ln 5}{\pi } &<&\frac{1}{r^{2}}\left( \frac{\ln \left(
1+4/r^{\prime }\right) }{\mathcal{K}\left( r\right) }-\frac{2\ln 5}{\pi }%
\right) <\frac{2}{\pi }\left( \frac{2}{5}-\frac{1}{4}\ln 5\right) .
\end{eqnarray*}
\end{remark}

\section{Acknowledgements.}

The authors would like to express their sincere thanks to the anonymous
referees for their great efforts to improve this paper.

\end{document}